\documentclass[a4paper,10pt]{article}
\usepackage{amsmath,amssymb}
\usepackage[utf8]{inputenc}
\usepackage[all]{xy}

\newtheorem{theorem}{Theorem}[section] 
\newtheorem{proposition}[theorem]{Proposition} 
\newtheorem{conjecture}[theorem]{Conjecture} 
\newtheorem{corollary}[theorem]{Corollary} 
\newtheorem{lemma}[theorem]{Lemma}

\newtheorem{remark}[theorem]{Remark}

\newcommand{\QQ}{\mathbb{Q}}

\newcommand{\prelie}{\operatorname{PreLie}}
\newcommand{\prelieh}{\operatorname{CycPreLie}}
\newcommand{\lie}{\operatorname{Lie}}
\newcommand{\lieh}{\operatorname{CycLie}}

\newcommand{\freeop}{\operatorname{FreeOp}}
\newcommand{\comm}{\operatorname{Comm}}
\newcommand{\ass}{\operatorname{Assoc}}
\newcommand{\perm}{\operatorname{Perm}}
\newcommand{\refl}{\operatorname{Reflex}}

\newcommand{\indec}{\operatorname{Indec}}
\newcommand{\rela}{\operatorname{Rela}}

\newcommand{\sym}{\mathfrak{S}}

\newcommand{\sous}{\triangleleft}

\newcommand{\bp}[2]{\langle #1\,,\,#2 \rangle}

\newcommand{\bl}[2]{( #1\,,\,#2 )}

\newcommand{\hada}{\odot}

\newenvironment{proof}{\begin{trivlist}\item{\bf{Proof.}}}
  {\hfill\rule{2mm}{2mm}\end{trivlist}}

\title{Fine structures inside the PreLie operad} \author{F. Chapoton} \date{\today}

\begin{document}

\maketitle

\begin{abstract}
  This article aims at a detailed analysis of the $\prelie$ operad. We
  obtain a more concrete description of the relationship between the
  anticyclic structure of $\prelie$ and the generators of $\prelie$ as
  a $\lie$-module, which was known before only at the level of
  characters. Building on this, we obtain an inclusion of the cyclic
  Lie module in the $\prelie$ operad. We conjecture that the image of
  this inclusion generates an interesting free sub-operad.
\end{abstract}

\setcounter{section}{-1}

\section{Introduction}

The aim of this article is to make a step forward in the study a
specific operad, called the PreLie operad, and to unravel some
unexpected relations with the Lie operad. The PreLie operad describes
pre-Lie algebras, which have been useful in various contexts, and were
introduced independently in the works of Gerstenhaber
\cite{gerst_ring} in deformation theory and in the works of Vinberg
\cite{vinberg} on convex homogeneous cones. The PreLie operad itself
has been described in \cite{chapotonlivernet1} in terms of
combinatorial objects called rooted trees. This description is
strongly connected to the use of rooted trees in numerical analysis
\cite{butcher}, because the space of vector fields on the affine space
is naturally a pre-Lie algebra.

Many results are already known on the PreLie operad. It was proved in
\cite{chapotonlivernet1} that it is a Koszul operad, and this result
has been reproved since then by other methods \cite{vallette_poset,hoffbeck}. Another
important property is the fact, first proved by Foissy in \cite{foissy},
that free pre-Lie algebras are free Lie algebras, where the Lie
bracket comes from antisymmetrising the pre-Lie product. Another proof
has been obtained in \cite{free_chapoton}, in the setting of operads rather than
algebras. This property translates into the fact that the PreLie operad
is a free left Lie-module, so that there exists a factorisation
\begin{equation*}
  \prelie = \lie \circ X
\end{equation*}
in the category of $\sym$-modules, where $\circ$ is the composition of $\sym$-modules. Our principal aim is to understand
better the $\sym$-module $X$.

Surprisingly enough, the $\sym$-module $X$ is related to another
structure on the PreLie operad. It was shown in \cite{chap_anticyclic}
that PreLie is an anticyclic operad. This means in particular that the
$\sym$-module $\prelie$ admits a primitive $\prelieh$, which means
that $\prelieh(n+1)\simeq \prelie(n)$ as modules over the symmetric
group $\sym_n$. It has been proved in \cite[Th. 5.3]{free_chapoton}
that for $n\geq 2$ there is an isomorphism of modules over the
symmetric group $\sym_n$,
\begin{equation*}
  X(n)\simeq \refl(n) \otimes \prelieh(n),
\end{equation*}
where $\refl(n)={\QQ^n}/{\QQ}$ is the quotient of the natural action
of $\sym_n$ on $\QQ^n$ by the diagonal subspace. This was proved using
explicit formulas for the characters of these modules. The main result
of the present article is to better explain this relation, by the
definition of an explicit isomorphism. This will be achieved in section
\ref{proofiso}. 

The other main result of the article is the following. It is well
known that the Lie operad is a cyclic operad \cite{kont_formal}, so that in
particular there exists a primitive $\lieh$ to the $\sym$-module
$\lie$. We obtain, using the previous parts of the article, an
injective morphism of $\sym$-modules from $\lieh$ to $\prelie$.

We expect that this inclusion of $\lieh$ into $\prelie$ should have
very nice properties. We conjecture that it generates a free
sub-operad of $\prelie$. A recent result by N. Bergeron and Loday
\cite{bergeron_loday} would be a direct corollary of this more general
conjecture. Moreover, we conjecture that the sub-operad of $\prelie$
generated by the image of $\lieh$ should be isomorphic to $X$ as a
$\sym$-module and should give a distinguished space of generators of
$\prelie$ as a $\lie$-module.

\section{General setting and notations}

We will use the description of operads using $\sym$-modules over
$\QQ$. The reader may want to consult \cite{stasheff_markl} for
further details on the theory of operads.

Recall that the category of $\sym$-modules over $\QQ$ is the category
of functors from the groupoid of finite sets to the category of
$\QQ$-vector spaces. By choosing an equivalent model for the groupoid
of finite sets, $\sym$-modules can also be described as collections of
vector spaces $P(n)$ with an action of the symmetric group $\sym_n$ on
$P(n)$ for $n \geq 0$. We will freely use both descriptions.

We will denote $S_1$ the $\sym$-module corresponding to the trivial
representation of $\sym_1$.

The derivative $P'$ of an $\sym$-module $P$ is defined as follows. For
every $n$, the space $P'(n-1)$ is $P(n)$, with the action of the
symmetric group $\sym_{n-1}$ given by restriction of the action of the
symmetric group $\sym_n$ on $P(n)$. A primitive of a $\sym$-module $P$
is a $\sym$-module $Q$ such that $Q'=P$.

The category of $\sym$-modules is endowed with a nonsymmetric tensor
product $\circ$ which is called the composition. An \textbf{operad} is
a monoid for this tensor product. We will also consider right and left
modules over operads, for the same monoidal structure.

We will also need the symmetric tensor structure $\otimes$ on the
category of $\sym$-modules defined by
\begin{equation*}
  (P \otimes Q)(I)=\bigoplus_{I=J \sqcup K} P(J)\otimes Q(K).
\end{equation*}
Using this tensor product, one can also define exterior and symmetric
powers of $\sym$-modules.

There is a third tensor product on the category of $\sym$-modules,
sometimes called the Hadamard product, defined by
\begin{equation*}
  (P \hada Q)(I)= P(I)\otimes Q(I).
\end{equation*}

There is an equivalent way to describe an operad $\mathcal{P}$, using
partial compositions instead of the global composition map
$\mathcal{P} \circ \mathcal{P} \rightarrow \mathcal{P}$. We will
denote by $\circ_*$ the partial composition at position $*$ in an
operad. Symbols like $*,\#$ will sometimes appear as placeholders for
positions of compositions.

% The following typographic conventions will hold. 
% \begin{itemize}
% \item Letters $a,b,c,d,\dots$ for indices in finite sets.
% \item Letters $u,v,w,x,y,z,\dots$ for elements in $\lie$ or $\prelie$.
% \item Greek letters for morphisms.
% \item Capital letters $R,S,T,\dots$ for rooted trees.
% \end{itemize}

\section{The $\lie$ operad}

Recall that a Lie algebra is a vector space $V$ endowed with an
antisymmetric bracket $(x,y) \mapsto [x,y]$ such that
\begin{equation}
  \label{jacobi}
  [x,[y,z]]+[x,[y,z]]+[x,[y,z]]=0.
\end{equation}
This relation is called the Jacobi identity.

Let $\lie$ be the operad describing Lie algebras.

The operad $\lie$ admits a presentation by generators and relations,
which amounts to the axiomatic description of Lie algebras given
above. There is an antisymmetric generator in two variables, and a
relation between compositions of two generators, given by \eqref{jacobi}.

Recall now that a symmetric bilinear form $\bl{}{}$ on a Lie algebra is
called invariant if
\begin{equation}
  \label{rule_lie}
  \bl{[x,y]}{z}=\bl{x}{[y,z]}.
\end{equation}
This invariance condition is classical and comes from the natural
invariance condition for bilinear forms under group actions.

This notion of invariant bilinear form leads to a structure of cyclic
operad on $\lie$, first introduced by Kontsevich
\cite{kont_formal}. The $\sym$-module $\lieh$ (cyclic Lie module) is
the right $\lie$-module defined by the exact sequence of right
$\lie$-modules
\begin{equation}
  K_L \to S^2 \lie \stackrel{\bl{}{}}{\longrightarrow} \lieh \to 0,
\end{equation}
where $K_L$ is the sub right $\lie$-module of $S^2 \lie$ generated by
the relation \eqref{rule_lie}.

As an $\sym$-module, $\lieh$ is a primitive of $\lie$.

\begin{remark}
  The $\sym$-module $\lieh$ is also called the Whitehouse module.
\end{remark}

The following statement is a general property of cyclic operads.

\begin{proposition}
  Let $I$ be a finite set. For every element $\ell$ in $\lieh(I)$ and
  every element $i \in I$, there is a unique element $\ell_i$ in $\lie(I
  \setminus \{i\})$ such that
  \begin{equation}
    \ell=\bl{i}{\ell_i}.
  \end{equation}
\end{proposition}

% utile pour Prop. 8.2.

For example, if $\ell=\bl{[x,y]}{z}$ and $i=x$ then $\ell_i=[y,z]$ by
\eqref{rule_lie}.

\section{The $\prelie$ operad}

Recall that a pre-Lie algebra is a vector space $V$ endowed with an
bilinear map $(x,y) \mapsto x \sous y$ such that
\begin{equation}
  \label{axiom_prelie}
  (x \sous y) \sous z - x \sous (y \sous z) = (x \sous z) \sous y - x \sous (z \sous y). 
\end{equation}

Let $\prelie$ be the operad describing pre-Lie algebras.

The operad $\prelie$ admits a combinatorial description using rooted
trees \cite{chapotonlivernet1}. Let us recall this briefly here.

For a finite set $I$, a rooted tree on $I$ is a connected and simply
connected graph with vertex set $I$, together with a distinguished
vertex called the root. One can then orient every edge towards the
root. Then $\prelie(I)$ is the vector space spanned by rooted trees on
$I$. The composition of two rooted trees $S\circ_i T$ is the sum of
all rooted trees obtained from the disjoint union of $S \setminus
\{i\}$ and $T$ by adding one edge for every edge that was incident to
$i$ in $S$, as follows. Edges that were incoming at $i$ must keep the
same start and end at some vertex of $T$. The edge that was outgoing
at $i$ (if it exists) must keep the same end and start at the root
of $T$. The root of every tree in the sum is the unique vertex with no
outgoing edge.

The operad $\prelie$ admits a presentation by generators and relations,
which amounts to the axiomatic description of pre-Lie algebras given
above. There is a generator in two variables, and a
relation between compositions of two generators, given by \eqref{axiom_prelie}.

From the description above by rooted trees of the operad $\prelie$,
one can get the following rule for the product $ S \sous T$ of two
rooted trees $S$ and $T$: this is the sum of all rooted trees obtained
from the disjoint union of $S$ and $T$ by adding one edge from the
root of $T$ to a vertex of $S$. The root of each tree in this sum is
the root of $S$.

Recall now \cite[\S\,5.3]{chap_anticyclic} that an antisymmetric bilinear form $\bp{}{}$ on a pre-Lie algebra is
called invariant if
\begin{equation}
  \label{rule_prelie1}
  \bp{x}{y\sous z}=-\bp{z}{y \sous x}
\end{equation}
and
\begin{equation}
  \label{rule_prelie2}
  \bp{x}{y\sous z}=\bp{y}{x \sous z-z \sous x}.
\end{equation}
This invariance condition is less classical than its Lie analogue, but
has been used in the study of left-invariant affine and symplectic
structures on Lie groups. It is also related to the notion of
quasi-Frobenius Lie algebra.

Note that \eqref{rule_prelie1} is a consequence of
\eqref{rule_prelie2} and that these relations also imply the following
relation:
\begin{equation}
  \label{rule_prelie3}
  \bp{x}{y\sous z}+\bp{y}{z\sous x}+\bp{z}{x\sous y}=0.
\end{equation}

This notion of invariant bilinear form leads to a structure of
anticyclic operad on $\prelie$, introduced in \cite{chap_anticyclic}. The anticyclic
$\sym$-module $\prelieh$ is the right $\prelie$-module defined by the
exact sequence of right $\prelie$-modules
\begin{equation}
  K_{P} \to \Lambda^2 \prelie \stackrel{\bp{}{}}{\longrightarrow} \prelieh \to 0,
\end{equation}
where $K_P$ is the sub right $\prelie$-module of $\Lambda^2 \prelie$
generated by the relations \eqref{rule_prelie1} and \eqref{rule_prelie2}.

As an $\sym$-module, $\prelieh$ is a primitive of $\prelie$.

The following statement is a general property of anticyclic operads.

\begin{proposition}
  \label{def_gamma_prelie}
  Let $I$ be a finite set. For every $t\in\prelieh(I)$ and
  every $i \in I$, there is a unique $\Gamma_i(t) \in
  \prelie(I \setminus \{i\})$ such that
  \begin{equation}
    t=\bp{i}{\Gamma_i(t)}.
  \end{equation}
\end{proposition}

For example, if $t=\bp{x}{y\sous z}$ then $\Gamma_z(t)=- y
\sous x$ by \eqref{rule_prelie1}.

Here is a more complicated example. If $t=\bp{w}{x\sous (y \sous z)}$, then one can show that $\Gamma_z(t)=- y\sous (x \sous w)$.

\smallskip

Let us now state some lemmas for later use. For short, we will write
$a \in s$ as as an abreviation for the sentence $s \in \prelie(I)$ and
$a \in I$ for some finite set $I$. We will also say that $r$ and $s$
have disjoint indices if $r \in \prelie(I)$ and $s \in\prelie(J)$ for
some disjoint finite sets $I$ and $J$.

\begin{lemma}
  Let $r$,$s$,$t$ in $\prelie$ with disjoint indices. Let $a \in s$
  and $b \in r\sqcup s$ distinct from $a$. Then
  \begin{equation}
    \Gamma_b(r\wedge(s\circ_a{t}))=\Gamma_b(r\wedge s) \circ_a{t}.
  \end{equation}
\end{lemma}
\begin{proof}
  Indeed, one has
  \begin{multline}
    \bp{b}{\Gamma_b(r\wedge(s\circ_a{t}))}
=\bp{r}{s\circ_a{t}}
=\bp{r}{s}\circ_a{t}\\
=\bp{b}{\Gamma_b(r\wedge s)}\circ_a{t}
=\bp{b}{\Gamma_b(r\wedge s)\circ_a{t}}.
  \end{multline}
  Here one uses the definition of $\Gamma$ and the fact that
  $\prelieh$ is a right $\prelie$-module.
\end{proof}

\begin{lemma}
  \label{lemma_gamma_2}
  Let $r$,$s$,$t$ in $\prelie$ with disjoint indices. Let $a \in s$
  and $b \in t$. Then
  \begin{equation}
    \Gamma_b(r\wedge(s\circ_a{t}))=-\Gamma_b(\#\wedge t) \circ_\# \Gamma_a(r\wedge s).
  \end{equation}
\end{lemma}
\begin{proof}
  Let us compute $\bp{\#}{t}\circ_{\#} \Gamma_a(r\wedge s)$ in two ways.
  On the one hand, this is equal by definition of $\Gamma$ to
  \begin{equation*}
    \bp{b}{\Gamma_b(\# \wedge t)}\circ_{\#} \Gamma_a(r\wedge s),
  \end{equation*}
  which can be rewritten as
  \begin{equation*}
    \bp{b}{\Gamma_b(\# \wedge t) \circ_{\#} \Gamma_a(r\wedge s)}.
  \end{equation*}
  On the other hand, this is equal by properties of $\Gamma$ to
  \begin{equation*}
    \bp{\Gamma_a(r\wedge s)}{t}=\bp{\Gamma_a(r\wedge s)}{a}\circ_a{t},
  \end{equation*}
  which by definition of $\Gamma$ is
  \begin{equation*}
    -\bp{r}{s}\circ_a{t}=-\bp{r}{s\circ_a{t}}=-\bp{b}{\Gamma_b(r\wedge(s\circ_a{t}))}.
  \end{equation*}
  This proves the expected equality.
\end{proof}

\section{$\prelie$ is free as a $\lie$-module}

Recall that every pre-Lie algebra is also a Lie algebra for the bracket defined by 
\begin{equation}
  [x,y]=x \sous y -y \sous x.
\end{equation}

This defines a morphism $\varphi$ of operads from $\lie$ to
$\prelie$. The composition with the projection from $\lie$ to the
associative operad $\ass$ is the usual inclusion of $\lie$ in $\ass$,
hence $\varphi$ is also injective.

From this morphism, one can deduce by restriction of composition in
$\prelie$ a structure of left $\lie$-module on $\prelie$:
\begin{equation}
  \lie \circ \prelie \stackrel{\gamma}{\longrightarrow} \prelie.
\end{equation}
It is clear from this definition that $\gamma$ is a morphism of right
$\prelie$-modules.

Let $\lie_{\geq 2}$ be the restriction of $\lie$ to degrees at least
$2$. One can restrict $\gamma$ to $\lie_{\geq 2} \circ \prelie$ and
define a $\sym$-module $\indec$ by the exact sequence of right
$\prelie$-modules
\begin{equation}
  \label{defi_gen}
  \lie_{\geq 2} \circ \prelie \stackrel{\gamma}{\longrightarrow} \prelie \stackrel{\pi}{\longrightarrow} \indec \to 0.
\end{equation}

The $\sym$-module $\indec$ was denoted by $X$ in the introduction.

It has been shown in \cite{free_chapoton} (see also \cite{foissy,bergeron_livernet}) that $\prelie$ is a free left
$\lie$-module: there exists an isomorphism of $\sym$-modules
\begin{equation}
  \prelie \simeq \lie \circ \indec.
\end{equation}
This isomorphism is not canonical, but depends on the choice of a
section of the projection $\pi$.

This statement of freeness can be reformulated as follows. The left
$\lie$-module structure of $\prelie$ can be considered as a structure
of $\lie$-algebra in the category of $\sym$-modules (with respect to
the tensor product $\otimes$). The usual theory of $\lie$ algebras
over a field has a natural extension to this setting, including
universal enveloping algebras and the Chevalley-Eilenberg complex (see
for instance \cite{stover}). Freeness as a left-$\lie$-module then
translates into freeness as a $\lie$ algebra, which implies that the
Chevalley-Eilenberg complex has homology only in degree $0$.

As the image of the leftmost arrow $\gamma$ of \eqref{defi_gen} is spanned by
the linear combinations of brackets of rooted trees, there is a long exact sequence of right $\prelie$-modules
\begin{equation}
  \cdots \to \Lambda^3 \prelie \stackrel{\delta}{\longrightarrow} \Lambda^2 \prelie \stackrel{\delta}{\longrightarrow} \prelie \stackrel{\pi}{\longrightarrow} \indec \to 0,
\end{equation}
where $\delta$ are the Chevalley-Eilenberg differentials. The rightmost $\delta$
sends $s \wedge t$ to $s \sous t - t \sous s$.

This gives a short exact sequence of right $\prelie$-modules
\begin{equation}
  \label{ses_big}
  0 \to \rela \stackrel{\delta}{\longrightarrow} \prelie \stackrel{\pi}{\longrightarrow} \indec \to 0,
\end{equation}
where $\rela$ is the quotient of $\Lambda^2 \prelie$ by the image of $\delta$.

\begin{proposition}
  \label{presentation6}
  The right $\prelie$-module $\rela$ is the quotient of $\Lambda^2
  \prelie$ by the sub-right $\prelie$-module generated by the
  following $6$-terms relations:
\begin{equation}
  \label{six_termes}
  r \wedge (s \sous t)+ s \wedge (t \sous r)+ t \wedge (r \sous s)
  -s \wedge (r \sous t)- t \wedge (s \sous r)- r \wedge (t \sous s)=0.
\end{equation}
\end{proposition}
\begin{proof}
  We are dealing here with the first few terms of the Chevalley-Eilenberg
  complex for the Lie algebra $\prelie$. As we know that this is a
  free $\lie$ algebra, there is no homology but in degree $0$. This implies that
  the kernel is spanned by the image of $r\wedge s \wedge t$ by $\delta$:
  \begin{equation}
    [r,s] \wedge t + [s,t]\wedge r -[r,t]\wedge s.
  \end{equation}
  Using the link between the bracket and $\sous$, this gives the
  expected relations.
\end{proof}

\section{Reduction to root-valence $1$}

Let us define the \textbf{root-valence} of a rooted tree $T$ to be the
valence of the root of $T$, \textit{i.e.} the number of edges adjacent
to the root.

Recall the following standard notation: for rooted trees
$T_1,\dots,T_k$, let 
$$B^+_{a}(T_1,T_2,\dots,T_k)$$
be the rooted tree
obtained by grafting $T_1,\dots,T_k$ on a new root with index $a$.

If $S$ is a rooted tree in $\prelie$, we will denote $[S]$ the class
of $S$ modulo the image of $\delta$, \textit{i.e.} modulo Lie brackets
of rooted trees.

\begin{proposition}
  \label{right_prelie_module}
  For every rooted trees $S$ and $T$ and $* \in S$, one has
  \begin{equation}
    [S \circ_* T]=[S] \circ_* T.
  \end{equation}
\end{proposition}
\begin{proof}
  This is because \eqref{ses_big} is an exact sequence of right
  $\prelie$-modules.
\end{proof}

\begin{proposition}
  \label{reduction_to_1}
  Every rooted tree $T$ in $\prelie$ (with at least two vertices) is
  equivalent modulo $\delta(\Lambda^2 \prelie)$ to a linear
  combination of rooted trees of root-valence $1$.
\end{proposition}

The proof is by induction on the size of rooted trees and uses several
lemmas, of increasing generality.

\begin{lemma}
  \label{step1}
  The statement is true if $T$ has root-valence $2$ and at least one
  of the two subtrees of the root is a singleton $a$.
\end{lemma}
\begin{proof}
  Let $T'$ be the tree $T$ with vertex $a$ removed. Then $T'$ has
  root-valence $1$. Then one has $[T',a]=T'\sous a - a \sous T'$. The
  tree $a \sous T'$ has root-valence $1$. One has $T'\sous a=T+r$,
  where $r$ is a sum of trees of root-valence $1$.
\end{proof}

\begin{lemma}
  \label{step2}
  The statement is true if $T$ has a leaf $a$ such that the vertex $b$
  under $a$ has valence $1$.
\end{lemma}
\begin{proof}
  In this case, one can write $T=T' \circ_* (ba)$, where $T'$ is a
  smaller tree with a leaf $*$. By induction on the size, the tree
  $T'$ is equivalent to a linear combination $\sum_{\alpha}
  \mathbf{c}_{\alpha} T_\alpha$ of trees of root-valence $1$. By
  Prop. \ref{right_prelie_module}, the tree $T$ is equivalent to the
  linear combination $\sum_{\alpha} \mathbf{c}_{\alpha} T_\alpha
  \circ_* (ba)$.

  If $*$ is not the root of $T_\alpha$, then $T_\alpha \circ_* (ba)$
  is a sum of trees of root-valence $1$. If $*$ is the root of
  $T_\alpha$, then $T_\alpha \circ_* (ba)$ is the sum of a rooted tree
  of root-valence $1$ plus a rooted tree of root-valence $2$, which
  satisfies the hypothesis of Lemma \ref{step1}.
\end{proof}

\begin{lemma}
  \label{step3}
  The statement is true if at most one of the subtrees of the root of
  $T$ is not a leaf.
\end{lemma}
\begin{proof}
  By induction on the root-valence $k$. If $k=1$, there is nothing to
  prove. If $k=2$, we can use Lemma \ref{step1} above.

  Assume now that $k$ is at least $2$ and write
  $T=B^+_{a}(T_1,a_2,\dots,a_k)$.

  Let $T'=B^+_{a}(T_1,a_2,\dots,a_{k-1})$ be
  the tree obtained from $T$ by removing the leaf $a_k$.

  Then $[T',a_k]=T' \sous a_k - a_k \sous T'$. The tree $a_k \sous T'$
  has root-valence $1$. One has $T' \sous a_k=T+r$, where $r$ is a sum
  of two kinds of trees: either $a_k$ is grafted on one of the leaves
  $a_i$, in which case one can apply Lemma \ref{step2}, or $a_k$ is
  grafted on $T_1$, in which case one can use the induction on $k$.
\end{proof}

\begin{lemma}
  \label{step4}
  The statement is true for any rooted tree $T$ with at least $2$ vertices.
\end{lemma}
\begin{proof}
  Pick in $T$ a vertex $a$ of maximal height, where the height of a
  vertex is the number of edges in the unique path to the root. Let
  $b$ be the vertex under $a$. The set of all vertices over $b$ is
  then a corolla $C_b$.

  In this case, $T$ can be written $T' \circ_* C_b$. By induction
  on the size, $T'$ is equivalent to a linear combination $\sum_{\alpha}
  \mathbf{c}_{\alpha} T_\alpha$ of
  trees of root-valence $1$.

  By Prop. \ref{right_prelie_module}, $T$ is equivalent to the linear
  combination $\sum_\alpha \mathbf{c}_{\alpha} T_\alpha \circ_* C_b$.

  If $*$ is not the root of $T_\alpha$, then $T_\alpha \circ_* C_b$ is
  a sum of trees of root-valence $1$. If $*$ is the root of
  $T_\alpha$, then $T_\alpha \circ_* C_b$ is a sum of rooted trees
  which satisfies the hypothesis of Lemma \ref{step3}.
\end{proof}

This concludes the proof of Prop. \ref{right_prelie_module}.

\begin{remark}
  What we used in the proof of Lemma \ref{step4} is a top corolla,
  \textit{i.e.} a vertex $b$ with only leaves above it. Instead of
  choosing one such vertex of maximal height, it may be more
  convenient for purposes of practical computation of $\rho$ to choose
  one with the smallest number of attached leaves.
\end{remark}

Let $\prelie_{v=1}$ be the sub $\sym$-module of $\prelie$ spanned by
rooted trees with root-valence $1$.

\begin{remark}
  The $\sym$-module $\prelie_{v=1}$ is not a right $\prelie$-module.
\end{remark}

Let $\rela_{v=1}$ be the subspace of $\rela$ that is mapped by $\delta$ to
$\prelie_{v=1}$.

Let $\indec_{\geq 2}$ be the sub-$\sym$-module of $\indec$ obtained by
removing the degree $1$ component of $\indec$.

From Prop. \ref{reduction_to_1} and the short exact sequence
\eqref{ses_big}, one obtains a short exact sequence of $\sym$-modules
\begin{equation}
  \label{ses_small}
  0 \to \rela_{v=1} \stackrel{\delta}{\longrightarrow} \prelie_{v=1} \stackrel{\pi}{\longrightarrow} \indec_{\geq 2} \to 0.
\end{equation}

Let us now describe another short exact sequence, and then compare them.

\section{A simple short exact sequence}

Let $\comm$ be the underlying $\sym$-module of the commutative
operad. For every finite set $I$, $\comm(I)=\QQ$.

Let $\perm$ be the underlying $\sym$-module of the permutative
operad. For every finite set $I$, $\comm(I)=\QQ I$.

There is an inclusion $\iota$ from $\comm$ to $\perm$ that sends $1 \in \comm(I)$ to $\sum_{i\in I} i \in \perm(I)$.

Let $\refl$ be the quotient $\sym$-module, so that there is a short exact sequence
\begin{equation}
  \label{suite_exacte_simple}
   0 \to \comm \stackrel{\iota}{\longrightarrow} \perm \stackrel{p}{\longrightarrow} \refl \to 0.
\end{equation}

By (Hadamard) tensor product with $\prelieh$, one gets a short exact sequence
\begin{equation}
  \label{ses2}
   0 \to \prelieh \stackrel{\iota}{\longrightarrow} \perm\hada\prelieh \stackrel{p}{\longrightarrow} \refl\hada\prelieh \to 0.
\end{equation}

There exists a section of the projection map $p$ in the short exact
sequence \eqref{suite_exacte_simple}, that maps the class of $i-j$ to
$i-j$. This gives a similar section of the projection $p$ in the short
exact sequence \eqref{ses2}.

%One can also consider the subspace $\rera$ of $\perm$ generated by the
%differences $i-j$. The induced map from $\rera$ to $\refl$ is an
%isomorphism over $\QQ$. Over $\ZZ$, the image of $\rera(I)$ is a
%sub-lattice in $\refl(I)$, of index the cardinality of $I$.

% root lattice and weight lattice of type A

\section{Isomorphism of exact sequences}

\label{proofiso}

In this section, we will obtain the following isomorphism of short
exact sequences:
\begin{equation}
  \xymatrix{
    0 \ar[r] & \prelieh \ar[d]^{\rho}\ar[r]^{\iota} & \perm\hada\prelieh \ar[d]^{\psi}\ar[r]^{p} & \refl\hada\prelieh \ar[d]^{\mu}\ar[r] & 0 \\
    0 \ar[r] & \rela_{v=1} \ar[r]^{\delta} & \prelie_{v=1} \ar[r]^{\pi} & \indec_{\geq 2} \ar[r] & 0 
}
\end{equation}

\subsection{Middle column}

\label{milieu_iso}

Let us start with the isomorphism between middle terms. There is a
simple isomorphism $\psi$ from $\perm\hada \prelieh$ to
$\prelie_{v=1}$ defined for $t$ in $\prelieh$ and $a \in t$ by
\begin{equation}
  \psi(a \otimes t) = a \sous \Gamma_a(t),
\end{equation}
where $\Gamma_a(t)$ is the function introduced in Prop. \ref{def_gamma_prelie}.

The inverse morphism maps a rooted tree of root-valence $1$, written
as $a \sous T$ to the expression $a \otimes \bp{a}{T}$.

Note that the composite morphism $\psi \iota$ has therefore the
following description:
\begin{equation}
  \psi(\iota(t))= \sum_{j} j\sous \Gamma_j(t).
\end{equation}

\subsection{Left column (up)}

\begin{proposition}
  There exists a morphism $\bp{}{}$ from $\rela$ to $\prelieh$.
\end{proposition}
\begin{proof}
  Let us start with the morphism $\bp{}{}$ from $\Lambda^2 \prelie$ to
  $\prelieh$ that defines the anticyclic structure. By
  Prop. \ref{presentation6}, it is enough to check that the $6$-terms
  relations \eqref{six_termes} are mapped to $0$. These relations are
  mapped in $\prelieh$ to
  \begin{equation*}
    \bp{r}{s \sous t}+ \bp{s}{t \sous r}+ \bp{t}{r \sous s}-
    \bp{s}{r \sous t}- \bp{t}{s \sous r}- \bp{r}{t \sous s}.
  \end{equation*}
  Using \eqref{rule_prelie3} twice, one obtains that this vanishes in
  $\prelieh$.
\end{proof}

By restriction of the morphism $\bp{}{}$ from $\rela$ to $\prelieh$, one has a
morphism from $\rela_{v=1}$ to $\prelieh$, still denoted $\bp{}{}$.

\subsection{Left column (down)}

\begin{theorem}
  \label{defi_rho}
  There exists a unique morphism $\rho$ from $\prelieh$ to
  $\rela_{v=1}$ such that $\delta \rho$ equals $ \psi \iota$, \textit{i.e.}
  \begin{equation}
    \label{carre}
    \delta(\rho(x))=\sum_{i} i \sous \Gamma_i(x).
  \end{equation}
  The morphism $\rho$ has the following property:
  \begin{equation}
    \label{property_z}
    \rho((r \wedge s)\circ_i t)=\rho(r\wedge s)\circ_i t-\rho(\# \wedge t) \circ_\# \Gamma_i(r\wedge s)+\Gamma_i(r\wedge s)\wedge t,
  \end{equation}
  for $r,s,t$ in $\prelie$ with disjoint indices and $i \in r \sqcup s$.
\end{theorem}

\begin{proof}
  Let us start by remarking that the uniqueness of $\rho$ is clear
  because $\delta$ is an injection.

  Assuming now for a moment that $\rho$ has been defined, let us prove the
  last statement, using uniqueness.  Let us apply $\delta$ to the
  right side of \eqref{property_z} :
  \begin{equation}
    \sum_{j\in r,s} (j \sous \Gamma_j(r\wedge s)) \circ_i t- \sum_{j\in \# \sqcup t} (j \sous \Gamma_j(\# \wedge t)) \circ_\# \Gamma_i(r\wedge s)+ [\Gamma_i(r\wedge s), t].
  \end{equation}
  The term with $j=i$ in the first sum and the term with $j=\#$ in the
  second sum annihilates with the bracket term. One can then use Lemma
  \ref{lemma_gamma_2} to rewrite the second sum and obtain
  \begin{equation}
    \sum_{ \genfrac{}{}{0pt}{}{j\in r,s}{j\not=i} } j \sous \Gamma_j((r\wedge s) \circ_i t)+\sum_{j\in t} j \sous \Gamma_j((r\wedge s) \circ_i t),
  \end{equation}
  which is exactly $\psi(\iota((r \wedge s)\circ_i t))$.

  \begin{remark}
    \label{cruciale}
    One can also interpret this argument as follows: if the equation
    \eqref{carre} is satisfied by the terms entering the right hand
    side of \eqref{property_z}, then it is also satisfied by the left
    hand side.
  \end{remark}

  Let us now enter the existence proof of $\rho$.

  It is enough to define a morphism from $\Lambda^2 \prelie$ to
  $\rela_{v=1}$ satisfying \eqref{carre}, as it will then
  automatically pass to the quotient $\prelieh$.

  As the set of elements $T \wedge T'$ (for rooted trees $T$ and $T'$
  with disjoint indices) spans $\Lambda^2 \prelie$, it is sufficient
  to define $\rho(T \wedge T')$.

  The definition is by induction on the cardinality of the finite set $I$,
  and, at fixed cardinality, by a four-steps process for pairs of trees
  of increasing generality.

  \textsc{First Step.}

  Let us define $\rho$ when $I$ has cardinality $2$ or $3$ as
  \begin{equation}
    \label{defi_rho2}
    \rho(a \wedge b)=a \wedge b,
  \end{equation}
  and
  \begin{equation}
    \label{defi_rho3}
    \rho(a \wedge (b \sous c))=a\wedge (b\sous c)-c\wedge (b \sous a).
  \end{equation}
  One can easily check that indeed $\delta \rho = \psi \iota$ in these cases.

  \textsc{Second Step.}

  Assume now that $T=a$ and that $T'=b \sous T''$ where $T''$ has at
  least two vertices.

  One then defines $\rho(T \wedge T')$ using induction on the number of
  vertices by the following formula:
  \begin{equation}
    \label{defi_rho_step2}
    \rho(a\wedge (b \sous T''))=\rho(a \wedge (b\sous *))\circ_* T''-\rho(\#\wedge T'')\circ_\# \Gamma_*(a \wedge (b\sous *))+\Gamma_*(a \wedge (b\sous *))\wedge T''.
  \end{equation}
  As this formula is an instance of formula \eqref{property_z}, one
  deduces from Remark \ref{cruciale} that \eqref{carre} holds in this case.

%  On the one hand, one
%   has
%   \begin{equation}
%     \iota(\psi(a\wedge (b \sous T'')))=\sum_{j\in \{a,b\}\cup T''} j \sous \Gamma_j(a\wedge b \sous T'').
%   \end{equation}
%   On the other hand, one finds that the image by $\delta$ of the right hand
%   side of \eqref{defi_rho_step2} is
%   \begin{equation}
%     \sum_{j\in \{a,b,*\}} j\sous \Gamma_j(a \wedge b\sous *)\circ_* T''- \sum_{j\in \{\#\}\cup T} j \sous \Gamma_j(\#\wedge T'')\circ_\# \Gamma_*(a \wedge b\sous *)+[\Gamma_*(a \wedge b\sous *), T''].
%   \end{equation}
%   The terms with $j=*$ and $j=\#$ annihilates with the bracket. Then use Lemma \ref{lemma_gamma_2}

  \textsc{Third Step.}
  
  Assume now that $T=a$ and that $T'$ has root-valence at least
  $2$. One write $T'=B^+_b(T_1,T_2,\dots,T_k)$. Let
  $T''=B^+_*(T_2,\dots,T_k)$. In this case, one has
  \begin{equation}
    T'=T''\circ_* (b\sous T_1)-\sum_{\alpha} T_\alpha,
  \end{equation}
  where the sums runs over terms of smaller root-valence.
  
  One then defines $\rho(T \wedge T')$ using induction on the
  root-valence of $T'$ by the following formula:
  \begin{multline}
    \label{defi_rho_step3}
    \rho(a\wedge T')=\rho(a \wedge T'')\circ_* (b\sous T_1)-\rho(\# \wedge (b\sous T_1))\circ_{\#}\Gamma_{*}(a \wedge T'')\\+\Gamma_{*}(a \wedge T'')\wedge (b \sous T_1)-\sum_\alpha \rho(a \wedge T_\alpha)
  \end{multline}

  As this formula is an instance of formula \eqref{property_z}, one
  deduces from Remark \ref{cruciale} that \eqref{carre} holds in this case.

%   Let us now check that \eqref{carre} holds in this case. On the one
%   hand, one has
%   \begin{equation}
%     \delta(\rho(a\wedge T'))=\sum_{j\in \{a,b\}\cup T_1,T_2,\dots,T_k} j\sous \Gamma_j(a \wedge T')
%   \end{equation}
  
%   On the other hand, one finds that the image by $\delta$ of the right hand
%   side of \eqref{defi_rho_step3} is
%   \begin{multline}
%     \sum_{j\in \{a,*\}\cup T_2,\dots,T_k} j\sous \Gamma_j(a \wedge T'')\circ_* (b\sous T_1)-\sum_{j\in \{b,\#\}\cup T_1} j \sous \Gamma_j(\# \wedge (b\sous T_1))\circ_{\#}\Gamma_{*}(a \wedge T'')\\+[\Gamma_{*}(a \wedge T''), (b \sous T_1)]-\sum_{j\in  \{a,b\}\cup T_1,T_2,\dots,T_k}\sum_\alpha \rho(a \wedge T_\alpha)
%   \end{multline}
%   The terms with $j=*$ or $j=\#$ annihilates with the bracket and there remains
%   \begin{multline}
%     \sum_{j\in \{a\}\cup T_2,\dots,T_k} j\sous \Gamma_j(a \wedge T'')\circ_* (b\sous T_1)-\sum_{j\in \{b\}\cup T_1} j \sous \Gamma_j(\# \wedge (b\sous T_1))\circ_{\#}\Gamma_{*}(a \wedge T'')\\-\sum_{j\in  \{a,b\}\cup T_1,T_2,\dots,T_k}\sum_\alpha \rho(a \wedge T_\alpha)
%   \end{multline}

%   To conclude, explain why both sides are the same TO DO

%  This implies also that the definition does not depend on the choice of $T_1$.

  \textsc{Fourth Step.}

  Assume now that neither $T$ nor $T'$ is a singleton.
  
  One then defines $\rho(T \wedge T')$ by the following formula:
  \begin{equation}
    \label{defi_rho_step4}
    \rho(T\wedge T')=\rho(T\wedge *)\circ_* T'+\rho(\# \wedge T')\circ_\# T - T\wedge T'.
  \end{equation}

  As this formula is an instance of formula \eqref{property_z}, the
  equation \eqref{carre} holds in this case by Remark \ref{cruciale} .

\end{proof}

For example, one can compute in this way that
\begin{equation}
  \rho(\bp{a}{b\sous(c\sous d)})=a \wedge b\sous(c\sous d) - c \sous d \wedge b \sous a - d \wedge  c\sous(b\sous a).
\end{equation}

\begin{theorem}
  \label{rho_n}
  The morphism $\rho$ satisfies $\bp{}{}(\rho (x)) = (n-1)\, x$ for
  every $x$ in $\prelieh(I)$, where $n=|I|$.
\end{theorem}
\begin{proof}
  The property is easy to check for small $n$. By the proof of
  Th. \ref{defi_rho}, it is then enough to check that this property is
  preserved by the formula \eqref{property_z} in the following sense:
  if it is true for the terms entering the right-hand side, it is true
  for the left-hand side.

  Applying $\bp{}{}$ on the right-hand side of
  \eqref{property_z}, one finds
  \begin{equation}
    (|r| + |s| - 1) \bp{r}{s} \circ_i t - |t| \bp{\#}{t} \circ_\# \Gamma_i(r\wedge s)+ \bp{\Gamma_i(r\wedge s)}{t}.    
  \end{equation}
  This can be rewritten as
  \begin{equation}
    (|r| + |s| + |t| - 2) \bp{r}{s} \circ_i t,
  \end{equation}
  which is indeed the correct value for the left-hand side.
\end{proof}

\subsection{Proof of the isomorphism}

By Theorem \ref{defi_rho}, we therefore have a morphism
$(\rho,\psi,\mu)$ from the short exact sequence
\begin{equation}
  0 \to \prelieh \stackrel{\iota}{\longrightarrow} \perm\hada \prelieh \stackrel{p}{\longrightarrow} \refl\hada\prelieh \to 0
\end{equation}
to the short exact sequence
\begin{equation}
  \label{ses_small_bis}
  0 \to \rela_{v=1} \stackrel{\delta}{\longrightarrow} \prelie_{v=1} \stackrel{\pi}{\longrightarrow} \indec_{\geq 2} \to 0,
\end{equation}
where $\mu$ is defined as the quotient morphism from
$\refl\hada\prelieh$ to $\indec_{\geq 2}$.

\begin{proposition}
  The morphism $\rho$ is injective.
\end{proposition}
\begin{proof}
  This follows from Th. \ref{rho_n}.
\end{proof}

Note that is is necessary here to work over the field $\QQ$.

\begin{proposition}
  \label{iso_ses}
  The triple $(\rho,\psi,\mu)$ is an isomorphism of short exact sequences.
\end{proposition}
\begin{proof}
  We already know that $\psi$ is an isomorphism (\S\,\ref{milieu_iso})
  and that $\rho$ is injective.

  To conclude, it is enough to note that the dimensions of the
  right-most terms are the same, both given by
  \begin{equation}
    (n-1)^{n-1}
  \end{equation}
  for $n \geq 2$, where $n=|I|$. Therefore the dimensions of the left-most terms
  coincide. This implies that $\rho$ is an isomorphism, hence the
  statement.
\end{proof}

\begin{corollary}
  \label{retrouve}
  One therefore has an isomorphism $\mu$:
  \begin{equation}
    \indec \simeq S_1\oplus\refl\hada \prelieh.
  \end{equation}
\end{corollary}
\begin{proof}
  One just need to check what happens in degree $1$. The component of
  degree $1$ of the $\sym$-module $\indec$ is just $S_1$.
\end{proof}

This equivalence has been proved in \cite[Th. 5.3]{free_chapoton} by a
computation of characters. We have obtained here an isomorphism that
explains why the characters are the same.

\section{Inclusion of $\lieh$ in $\perm \hada \prelieh$}

Let us now define a map $\theta$ from $\lieh$ to $\perm \hada \prelieh$.

%An element of $\lieh$ is $\bl{u}{v}$, where $u$ and $v$ are in $\lie$.

As $\lie$ is a cyclic operad with $\lieh$ as cyclic structure, there
is a surjective morphism of right $\prelie$-modules
\begin{equation}
  S^2 \lie \to \lieh.
\end{equation}
We will first define $\theta$ on $S^2 \lie$ and then check that it is
well defined on the quotient $\lieh$.

From the morphism of operads $\varphi$ from $\lie \to \prelie$, one
has a map $S^2 \varphi$:
\begin{equation}
  S^2 \lie \to S^2 \prelie.
\end{equation}

Let $\ell$ be an element of $S^2 \lie$. Then one decomposes its image
by $S^2 \varphi$ in $S^2 \prelie$ according to the roots of the trees:
\begin{equation}
  S^2 \varphi(\ell)=\sum_{a,b}\mathbf{c}_{a,b} T_a T_b,
\end{equation}
where $T_\#$ denotes a tree with root $\#$.
Then one can define
\begin{equation}
  \theta(\ell)=\sum_{a,b} \mathbf{c}_{a,b} (a-b) \otimes \bp{T_a}{T_b} 
\end{equation}
with values in $\perm\hada \prelieh$.

%, $\refl$ as a subspace of $\perm$.

\begin{remark}
  \label{image_in_reflex}
  Let us note that the image of $\theta$ is contained in the image of
  the section from $\refl\hada \prelieh$ to $\perm\hada \prelieh$.
\end{remark}

% One can as well define $\theta'$
% \begin{equation}
%   \theta'(\ell)=\sum_{a,b} [a-b] \bp{T_a}{T_b} 
% \end{equation}
% into $\refl$ as a quotient of $\perm$.

\begin{proposition}
  The map $\theta$ is well defined on $\lieh$.
\end{proposition}
\begin{proof}
  One just has to check that relation \eqref{rule_lie} holds.

  Consider the image of $m_1 [m_2,m_3]$. This is
  \begin{equation}
    \label{lhs_theta}
    \sum_{a,b,c} (a-b)\otimes \bp{\varphi(m_1)_a}{\varphi(m_2)_b\sous \varphi(m_3)_c}-(a-c)\otimes \bp{\varphi(m_1)_a}{\varphi(m_3)_c\sous \varphi(m_2)_b}.
  \end{equation}
  Consider now the image of $[m_1,m_2] m_3$. This is
  \begin{equation}
    \sum_{a,b,c} (a-c)\otimes \bp{\varphi(m_1)_a\sous \varphi(m_2)_b}{\varphi(m_3)_c}-(b-c)\otimes \bp{\varphi(m_2)_b\sous \varphi(m1)_a}{\varphi(m_3)_c}.
  \end{equation}
  Then one can rewrite this using the anticyclic structure of $\prelieh$, thanks to \eqref{rule_prelie1} and \eqref{rule_prelie2}, to obtain
    \begin{multline*}
    \sum_{a,b,c} (a-c)\otimes \bp{\varphi(m_3)_c\sous \varphi(m_2)_b}{\varphi(m_1)_a}\\-(a-c)\otimes \bp{\varphi(m_2)_b\sous \varphi(m_3)_c}{\varphi(m_1)_a}\\+(b-c)\otimes \bp{\varphi(m_2)_b\sous \varphi(m_3)_c}{\varphi(m_1)_a}.
  \end{multline*}
  This is exactly \eqref{lhs_theta}.
\end{proof}

\begin{proposition}
  The map $\theta$ is injective. The composite map $p \theta$ is also injective.
\end{proposition}
\begin{proof}
  Let us prove the first statement.

  Let us fix a finite set $I$ and let $x$ be an element of $\lieh(I)$
  in the kernel of $\theta$. Let us also choose an element $i\in
  I$. There is a unique element $y$ of $\lie(I \setminus \{i\})$ such
  that
  \begin{equation}
    x=\bl{i}{y}.
  \end{equation}
  By definition, one has
  \begin{equation}
    \theta(x)=\sum_{j\not=i} (i-j) \otimes \bp{i}{\varphi(y)_j},
  \end{equation}
  where $\varphi(y)_j$ is the projection of $\varphi(y)$ on the span
  of trees with root $j$. The hypothesis $\theta(x)=0$ implies that
  for every $j\not=i$,
  \begin{equation}
    \varphi(y)_j=0.
  \end{equation}
  So we obtain
  \begin{equation}
    \sum_{j\not= i} \varphi(y)_j=0=\varphi(y).
  \end{equation}
  But $\varphi$ is injective, hence $y=0$ and $x=0$.

  The second statement follows from Remark \ref{image_in_reflex}.

\end{proof}

\section{Inclusion of $\lieh$ in $\prelie_{v=1}$ and conjectures}

Using the previous inclusion $\theta$ and Proposition \ref{iso_ses},
one gets a map $\lambda$ from $\lieh$ to $\prelie_{v=1}$, hence to
$\prelie$.

\begin{proposition}
  The map $\lambda$ is injective.
\end{proposition}
\begin{proof}
  This is because $\lambda$ is the composition of $\theta$
  (injective) and the isomorphism $\psi$ from $\perm\hada \prelieh$ to $\prelie_{v=1}$.
\end{proof}

Let $M$ be the suboperad of $\prelie$ generated by the image of
$\lieh$ by $\lambda$ in $\prelie$.
\begin{conjecture}
  \label{conj_subop_gen}
  The sub-operad $M$ of $\prelie$ is a module of generators of
  $\prelie$ as a free left $\lie$-module, \textit{i.e.}
  \begin{equation}
    \prelie \simeq \lie \circ M
  \end{equation}
  and
  \begin{equation}
    M \simeq \indec.
  \end{equation}
\end{conjecture}

\begin{conjecture}
  \label{conj_free_subop}
  The sub-operad $M$ of $\prelie$ is isomorphic to the free operad on
  $\lieh$, \textit{i.e.}
  \begin{equation}
    \freeop(\lieh) \simeq M.    
  \end{equation}
\end{conjecture}

In the direction of Conjecture \ref{conj_free_subop}, Loday and
N. Bergeron \cite{bergeron_loday} have proved that the suboperad of
$\prelie$ generated by $x \sous y + y \sous x$ is free. The element $x \sous y + y \sous x$ is the image by $\lambda$ of the element $\bl{x}{y}\in\lieh$.

Let us summarise these conjectures in words. The image of the inclusion of
$\lieh$ in $\prelie$ should generate a free sub-operad of
$\prelie$. This free sup-operad should give a distinguished set of
generators of $\prelie$ as a left $\lie$-module. One should therefore
have a \textbf{canonical} isomorphism
\begin{equation}
  \prelie \simeq \lie \circ \freeop(\lieh)
\end{equation}
and an isomorphism
\begin{equation}
  \freeop(\lieh) \simeq \indec \simeq S_1\oplus\refl\hada \prelieh,
\end{equation}
where the last isomorphism is Corollary \ref{retrouve}.

Let us explain now how conjecture \ref{conj_free_subop} would follow
from conjecture \ref{conj_subop_gen} by an argument of dimension.

Indeed, there exists a surjective
morphism of operads from the free operad $F$ on $\lieh$ to $M$, as $M$
is generated by $\lieh$. To prove that this is an isomorphism, it is
enough to compute the generating series. On the one hand, the
generating series $f_M$ of $M$ is defined by
\begin{equation}
  f_{\prelie}(x)=f_{\lie}(f_M(x))=-\log(1-f_M(x)).
\end{equation}
and therefore one has
\begin{equation}
  1-f_M(x)=\exp(-f_{\prelie}(x)).
\end{equation}

On the other hand, the generating series of the free operad $F$ on
$\lieh$ is the unique solution with zero constant term of the
fixed-point equation
\begin{equation}
  f_F(x)=x+f_{\lieh}(f_F(x))=x+(1-f_F(x))\log(1-f_F(x))+f_F(x).
\end{equation}
which amounts to
\begin{equation}
  x=(1-f_F(x))(-\log(1-f_F(x)).
\end{equation}

One can then use the equality
\begin{equation}
  f_{\prelie}(x)=x \exp(f_{\prelie}(x))
\end{equation}
to show that
\begin{equation}
  x=(1-f_M(x))(-\log(1-f_M(x)).
\end{equation} 
and deduce from this that one must have $f_M=f_F$.

\begin{remark}
  If these conjectures are true, they give a way to define a bigrading
  on the vector spaces $\prelie(n)$, where one degree comes from the
  free operad structure and the other from the free $\lie$-module
  structure. This seems to be related to classical polynomial
  analogues of $n^{n-1}$ and combinatorial statistics on rooted trees
  considered in \cite{dumont_ra}.
\end{remark}

\thanks{\textbf{Remerciements} : l'auteur a bénéficié du soutien du projet PEPS ``Calcul moulien et MuPAD-Combinat''.}

\bibliographystyle{plain}
\bibliography{oplibre}

\end{document}